\documentclass[letterpaper,11pt,reqno]{amsart}
\usepackage{vmargin,color}
\usepackage{amsmath,amsfonts,amsthm,graphicx,amssymb}
\usepackage{esint}
\usepackage{caption}
\newcommand{\E}{\mathcal{E}}
\newcommand{\F}{\mathcal F}

\renewcommand{\H}{\mathcal{H}}

\newcommand{\Sc}{\mathcal S}

\newcommand{\N}{\mathbb{N}}

\newcommand{\R}{\mathbb{R}}
\renewcommand{\SS}{\mathbb{S}}

\newcommand{\Om}{\Omega}

\newcommand{\g}{\gamma}
\newcommand{\de}{\delta}
\newcommand{\e}{\varepsilon}
\renewcommand{\k}{\kappa}
\renewcommand{\l}{\lambda}
\newcommand{\s}{\sigma}

\newcommand{\Div}{{\rm div}\,}

\newcommand{\dist}{{\rm dist}}

\newcommand{\weakstar}{\stackrel{*}{\rightharpoonup}}

\newcommand{\pa}{\partial}

\newcommand{\cc}{\subset\subset}

\newcommand{\cl}{\mathrm{cl}\,}

\newcommand{\U}{\mathcal{U}}
\newcommand{\C}{\mathcal{C}}

\newcommand{\KK}{\mathcal{K}}

\newcommand{\eps}{\varepsilon}

\newcommand{\conv}{\mathrm{conv}}

\newcommand{\mres}{\mathbin{\vrule height 1.6ex depth 0pt width %measure restriction
0.13ex\vrule height 0.13ex depth 0pt width 1.3ex}}

\theoremstyle{plain}% default
\newtheorem{theorem}{Theorem}[section]
\newtheorem{lemma}[theorem]{Lemma}

\newtheorem*{theorem*}{Theorem}
\newtheorem*{corollary*}{Corollary}

\theoremstyle{definition}

\newtheorem{definition}[theorem]{Definition}

\newtheorem{remark}[theorem]{Remark}
\newtheorem{assumption}[theorem]{Assumption}

\newtheorem*{notation*}{Notation}

\numberwithin{equation}{section}
\numberwithin{figure}{section}

\title{Collapsing and the convex hull property \\ in a soap film capillarity model}

\author{Darren King, Francesco Maggi, and Salvatore Stuvard}
\address{Department of Mathematics, The University of Texas at Austin, 2515 Speedway, Stop C1200, Austin TX 78712-1202, USA}
\medskip
\email{king@math.utexas.edu}
\medskip

\email{maggi@math.utexas.edu}
\medskip

\email{stuvard@math.utexas.edu}

\begin{document}

\begin{abstract} Soap films hanging from a wire frame are studied in the framework of capillarity theory. Minimizers in the corresponding variational problem are known to consist of positive volume regions with boundaries of constant mean curvature/pressure, possibly connected by ``collapsed'' minimal surfaces. We prove here that collapsing only occurs if the mean curvature/pressure of the bulky regions is negative, and that, when this last property holds, the whole soap film lies in the convex hull of its boundary wire frame.\\

\textsc{Keywords:} convex hull property, minimal surfaces, constant mean curvature surfaces, Plateau's problem.\\

\textsc{AMS Math Subject Classification (2010):} 49Q05 (primary), 53A10, 49Q20.
\end{abstract}

\maketitle

\section{Introduction} We continue the analysis, started in \cite{kms}, of the variational model for soap films spanning a wire frame introduced in \cite{maggiscardicchiostuvard}. In this {\bf soap film capillarity model}, soap films are described as three-dimensional regions of small volume, rather than as two-dimensional surfaces with vanishing mean curvature, i.e. as minimal surfaces. In \cite{kms} we have proved the existence of {\it generalized} minimizers in the soap film capillarity model. The term generalized indicates the possibility for minimizing sequences of three-dimensional regions to locally collapse onto two-dimensional surfaces. Correspondingly, a generalized minimizer consists: of a three-dimensional set enclosing the prescribed small volume of liquid, with boundary of constant mean curvature $\l$ -- where the value of $\l$ is proportional to the pressure of the soap film; and, possibly, of a two-dimensional surface with zero mean curvature, whose area has to be counted twice in computing the energy of the minimizer; see
\begin{figure}
  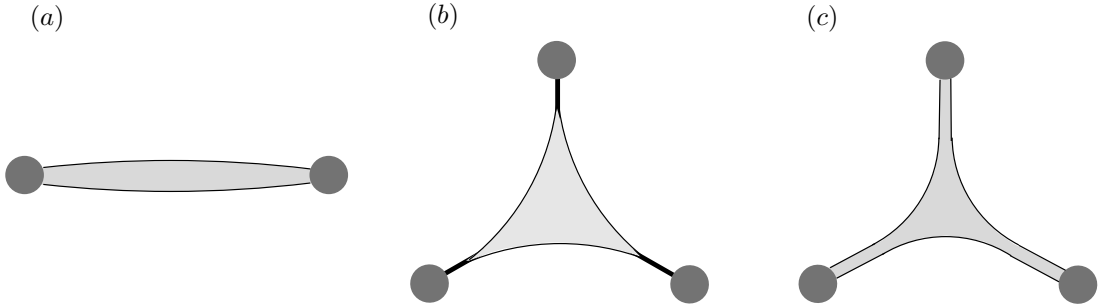
  \caption{{\small Generalized minimizers in the soap film capillarity model in the ``planar case'', where the ``boundary wire frame'' reduces to finitely many small disks (depicted in dark grey). We minimize the length of the boundary of two-dimensional regions, depicted in light gray, enclosing a given (small) volume $\e$ and spanning the boundary disks. (a) When the boundary consists of two disks, and $\e$ is small enough, we have a non-collapsed minimizing region bounded by two almost flat circular arcs of curvature $\l={\rm O}(\e)$. (b) When the boundary consists of three disks, and $\e$ is small enough, we have a collapsed minimizer given by a combination of a two-dimensional region bounded by circular arcs of negative curvature $\l=-{\rm O}(1/\sqrt{\e})$, and of three segments (depicted by thick lines) whose length has to be counted with double multiplicity to compute the minimizing energy. Collapsing corresponds with the situation, depicted in (c), where minimizing sequences consist of two-dimensional regions with opposite parts of their boundaries becoming increasingly closer to each other.}}
  \label{fig whatisthisabout}
\end{figure}
Figure \ref{fig whatisthisabout}. When this second possibility occurs, we speak of {\bf collapsed minimizers}. When collapsing does not occur, generalized minimizers are just regular minimizers, in the sense that they correspond to three-dimensional regions belonging to the  competition class. In this paper we prove two related results concerning important geometric properties of generalized minimizers, that can be roughly stated as follows:
\begin{enumerate}
  \item[(i)] if collapsing occurs, then the constant mean curvature/pressure $\l$ must be non-positive (Theorem \ref{thm lambda is negative});
  \item[(ii)] if $\l$ is non-positive, then the generalized minimizer is contained into the convex hull of the boundary wire frame (Theorem \ref{thm convex hull}); this convex hull property is of course a basic property of minimal surfaces, therefore the interest of establishing it in this setting.
\end{enumerate}
Theorem \ref{thm lambda is negative} is proved by comparing (through a technically delicate argument) a collapsed minimizer with competitors obtained by slightly de-collapsing its collapsed region (with a net increase of volume), followed by slightly deflating the bulky part of the minimizer (to restore the enclosed volume); see Figure \ref{fig openup} below. The proof of Theorem \ref{thm convex hull} is an adaptation to our context of the classical argument used to prove the convex hull property on stationary varifolds.

\medskip

The paper is organized as follows. In section \ref{section notation} we formally introduce the soap film capillarity model and state our main results (together with some necessary background results proved in \cite{kms}). Sections \ref{sec:proof1} and \ref{sec:proof2} contain, respectively, the proofs of  Theorem \ref{thm lambda is negative} and Theorem \ref{thm convex hull}. 

\medskip

\noindent{\bf Acknowledgments.} This work was supported by the NSF grants DMS-1565354, DMS-RTG-1840314, DMS-FRG-1854344, and DMS-2000034. S.S. acknowledges support from the AMS and the Simons Foundation through an AMS-Simons Travel Grant 2020.

\section{Statements}\label{section notation}

\subsection{Notation} The ambient space we will be working in is Euclidean space $\R^{n+1}$ with $n \geq 2$. For $A \subset \R^{n+1}$, $\cl(A)$ is the topological closure of $A$ in $\R^{n+1}$, $\conv(A)$ is its convex hull, and $I_\de(A), U_\de(A)$ are its closed and open $\de$-tubular neighborhoods, respectively. $B_r(x)$ is the open ball centered at $x \in \R^{n+1}$ with radius $r>0$. If $A$ is (Borel) measurable, $|A|$ and $\H^s(A)$ denote its Lebesgue and $s$-dimensional Hausdorff measure, respectively. We will adopt standard terminology in Geometric Measure Theory, for which we refer the reader to \cite{SimonLN,AFP,maggiBOOK}. In particular, given an integer $0 \leq k \leq n+1$, a Borel measurable set $M\subset\R^{n+1}$ is {\bf countably $k$-rectifiable} if it can be covered, up to an $\H^k$-negligible set, by countably many Lipschitz images of $\R^k$; it is {\bf (locally) $\H^k$-rectifiable} if it is countably $k$-rectifiable and, in addition, the $\H^k$ measure of $M$ is (locally) finite. A Borel set $E \subset \R^{n+1}$ is of {\bf locally finite perimeter} if there exists an $\R^{n+1}$-valued Radon measure $\mu_E$ on $\R^{n+1}$ such that $\langle \mu_E, X \rangle = \int_{E} \Div(X)\, dx$ for all vector fields $X \in C^1_c(\R^{n+1};\R^{n+1})$, and of {\bf finite perimeter} if $P(E) := |\mu_E|(\R^{n+1}) < \infty$. More generally, one can consider, for any Borel set $F \subset \R^{n+1}$, the quantity $|\mu_E|(F)$, which is called the {\bf relative perimeter} $P(E;F)$ of $E$ in $F$. The {\bf reduced boundary} of a set $E$ of finite perimeter is the set $\pa^*E$ of points $x \in {\rm spt} \, |\mu_E|$ such that $\left(|\mu_E|(B_r(x))\right)^{-1} \mu_E(B_r(x)) \to \nu_E(x)$ for some vector $\nu_E(x) \in \mathbb{S}^n$ as $r \to 0^+$. By De Giorgi's structure theorem, if $E$ has finite perimeter then $\pa^*E$ is $\H^n$-rectifiable, and the Gauss-Green measure $\mu_E$ and its total variation $|\mu_E|$ satisfy $\mu_E = \nu_E\, \H^n \mres \pa^*E$ and $|\mu_E| = \H^n \mres \pa^*E$, respectively.

\subsection{The soap film capillarity model}\label{problem}

Next, we recall the precise formulation of the variational problem introduced in \cite{kms}, and we outline the theory developed in there. We fix a compact set $W\subset\R^{n+1}$ (the ``wire frame''), and we denote the region accessible by the soap film as
\[
\Om:=\R^{n+1}\setminus W\,.
\]
The model scenario we have in mind is the physical case when $n+1 = 3$, and $W = I_\de(\Gamma)$ is the closed $\de$-neighborhood of a closed Jordan curve $\Gamma \subset \R^3$; nonetheless, admissible choice of $W$ will be more general than that. Following the Harrison-Pugh formulation of Plateau's problem \cite{harrisonpughACV,harrisonpughGENMETH} as extended in \cite{DLGM}, we introduce a {\bf spanning class} $\C$, that is, a non-empty family of smooth embeddings of $\SS^1$ into $\Om$ which is closed by homotopy in $\Om$, in the sense that if $\gamma \in \C$ and $\tilde \gamma$ is smooth and homotopically equivalent to $\gamma$ in $\Om$ \footnote{This means that there exists a continuous map $f \colon \left[0,1\right] \times \SS^1 \to \Om$ such that $f(0,\cdot) = \gamma$ and $f(1,\cdot) = \tilde \gamma$.} then $\tilde \gamma \in \C$. A set $S$ is {\bf $\C$-spanning $W$} if $S \cap \gamma \neq \emptyset$ for all $\gamma \in \C$. The (homotopic) {\bf Plateau's problem defined by $(W,\C)$} is then
\begin{equation}
  \label{plateau problem}
  \ell:=\inf\left\lbrace\H^n(S)\,:\,S\in\Sc\right\rbrace\,,
\end{equation}
where
\begin{equation} \label{class_S}
\Sc:=\left\lbrace S\subset\Om\,:\,\mbox{$S$ is relatively closed in $\Om$ and $S$ is $\C$-spanning $W$}\right\rbrace\,.
\end{equation}
The capillarity approximation \eqref{psi eps} of the Plateau's problem \eqref{plateau problem} has been studied in \cite{kms} under the following set of assumptions on $W$ and $\C$:
\begin{assumption} \label{ass:main}
The compact set $W$ and the spanning class $\C$ are such that the following holds:
\begin{enumerate}
\item[(A1)] Plateau's problem $\ell$ defined in \eqref{plateau problem} satisfies $\ell < \infty$; in particular, by \cite{harrisonpughACV,DLGM}, there exists a relatively compact, $\H^n$-rectifiable set $S \subset \Om$ such that $\H^n(S)=\ell$ \footnote{In addition, when $n=2$, every such minimizer $S$ is an Almgren-minimizer in $\Om$, and therefore satisfies Plateau's laws away from $W$ thanks to \cite{Almgren76,taylor76}. This result will not be needed in the sequel, but it is important because it establishes the physical relevance of the model.};
\item[(A2)] $\pa W = \pa \Om$ is a $C^2$-regular hypersurface in $\R^{n+1}$;
\item[(A3)] there exists $\tau_0>0$ such that, for every $\tau<\tau_0$, $\R^{n+1}\setminus I_{\tau}(W)$ is connected;
\item[(A4)] there exist $\eta_0>0$ and a minimizer $S$ in $\ell$ such that $\g\setminus I_{\eta_0}(S)\ne\emptyset$ for every $\g\in\C$.
\end{enumerate}
\end{assumption}

The conditions in Assumption \ref{ass:main} seem very reasonable towards the development of a theory of soap films, and are definitely valid in a reasonably large class of initial conditions. In fact, as a by-product of a technical result contained in the present paper, see Lemma \ref{lemma opening} below, one can see that all the results from \cite{kms} (and thus all the results of the present paper) still hold without the need of assuming (A4). This point is explained in detail in Section \ref{appendix remove A4} below.

\medskip

Next, we can define the capillarity problem $\psi(\eps)$ at volume $\eps>0$ as
\begin{equation}
  \label{psi eps}
  \psi(\e):=\inf\left\lbrace\H^n(\Om\cap\pa E)\,\colon\,\mbox{$E\in\E$, $|E|=\e$, $\Om\cap\pa E$ is $\C$-spanning $W$}\right\rbrace\,,
\end{equation}
where the competition class $\E$ is given by
\begin{equation}\label{class E}
  \E:=\left\lbrace E\subset\Om\,:\,\mbox{$E$ is an open set and $\pa E$ is $\H^n$-rectifiable}\right\rbrace\,.
\end{equation}
We explicitly observe that each $E \in \E$ is an open set of finite perimeter, and that $P(E;\Om) = \H^n(\Om \cap \pa^*E) \leq \H^n(\Om \cap \pa E)$. We also define the class
\begin{equation} \label{class K}
\begin{split}
  \KK:=
  \Big\{ (K,E)\,:\,&\mbox{$E\subset\Om$ is open with $\Om\cap\cl(\pa^*E)=\Om\cap\pa E\subset K$\,,}
  \\
  &\mbox{$K\in\Sc$ and $K$ is $\H^n$-rectifiable}\Big\}\,.
\end{split}
\end{equation}
For $(K,E) \in \KK$, its relaxed energy is given by
\begin{equation} \label{relaxed}
\F (K,E) := \H^n (\Om \cap \pa^*E) + 2\, \H^n (K \setminus \pa^* E)\,.
\end{equation}

We are now in the position to recall the main results from \cite{kms}, which lay the groundwork for the present analysis.

\begin{theorem}[Existence of generalized minimizers, see {\cite[Theorem 1.4]{kms}}]\label{thm lsc}
Let $W$ and $\C$ satisfy Assumption \ref{ass:main}, and let $\e > 0$. If $\{E_j\}_{j=1}^{\infty}$ is a minimizing sequence for $\psi(\e)$, then there exists a pair $(K,E)\in\KK$ with $|E|=\e$ such that, up to possibly extracting subsequences, and up to possible modifications of each $E_j$ outside a large ball containing $W$ (with both operations resulting in defining a new minimizing sequence for $\psi(\e)$, still denoted by $\{E_j\}_j$), we have that
  \begin{equation}\label{minimizing seq conv to gen minimiz}
    \begin{split}
    &\mbox{$E_j\to E$ in $L^1(\Om)$}\,,
    \\
    &\H^n\mres(\Om\cap\pa E_j)\weakstar \theta\,\H^n\mres K\qquad\mbox{as Radon measures in $\Om$}
    \end{split}
  \end{equation}
  as $j\to\infty$, for an upper semicontinuous multiplicity function $\theta:K\to\R$ satisfying
  \begin{equation}
    \label{theta density}
    \mbox{$\theta= 2$ $\H^n$-a.e. on $K\setminus\pa^*E$},\qquad\mbox{$\theta=1$ on $\Om\cap\pa^*E$}\,.
  \end{equation}
  Moreover, $\psi(\e)=\F(K,E)$ and, for a suitable constant $C$, $\psi(\e)\le 2\,\ell+C\,\e^{n/(n+1)}$.
\end{theorem}

\begin{definition} \label{def:gen_min}
A pair $(K,E) \in \KK$ with $|E| = \e$ is a {\bf generalized minimizer} for the capillarity problem $\psi(\e)$ if:
\begin{itemize}
\item[(a)] there exists a minimizing sequence $\{E_j\}_{j=1}^\infty$ of sets $E_j \in \E$ such that \eqref{minimizing seq conv to gen minimiz} holds for an upper semicontinuous function $\theta$ as in \eqref{theta density};
\item[(b)] $\F (K,E) = \psi(\e)$.
\end{itemize}
\end{definition}

\begin{theorem}[Euler-Lagrange equation for generalized minimizers, see {\cite[Theorem 1.6]{kms}}]\label{thm basic regularity}
 If $(K,E)$ is a generalized minimizer of $\psi(\e)$ and $f:\Om\to\Om$ is a diffeomorphism such that $|f(E)|=|E|$, then
 \begin{equation}
   \label{minimality KE against diffeos}
   \F(K,E)\le\F(f(K),f(E))\,.
 \end{equation}
 In particular:
 \begin{enumerate}
   \item[(i)] there exists $\l\in\R$ such that
  \begin{equation}
    \label{stationary main}
    \l\,\int_{\pa^*E}X\cdot\nu_E\,d\H^n=\int_{\pa^*E}\Div^K\,X\,d\H^n+2\,\int_{K\setminus\pa^*E}\Div^K\,X\,d\H^n
  \end{equation}
  for every $X\in C^1_{c}(\R^{n+1};\R^{n+1})$ with $X\cdot\nu_\Om=0$ on $\pa\Om$, where $\Div^K$ denotes the tangential divergence operator along $K$;
  \item[(ii)] there exists $\Sigma\subset K$, closed and with empty interior in $K$, such that $K\setminus\Sigma$ is a smooth hypersurface, $K\setminus(\Sigma\cup\pa E)$ is a smooth embedded minimal hypersurface, $\H^n(\Sigma\setminus\pa E)=0$, $\Om\cap(\pa E\setminus\pa^*E)\subset \Sigma$ has empty interior in $K$, and $\Om\cap\pa^*E$ is a smooth embedded hypersurface with constant scalar (w.r.t. $\nu_E$) mean curvature $\l$.
 \end{enumerate}
\end{theorem}

\begin{remark}

The conclusions about the regularity properties of the set $K$ achieved in Theorem \ref{thm basic regularity}(ii) are a straightforward consequence of the Euler-Lagrange equation \eqref{stationary main} and of Allard's regularity theorem for varifolds with bounded generalized mean curvature. A more refined analysis, which crucially relies on the structure of the variational problem $\psi(\e)$, was carried out in \cite{kms3}. A fundamental outcome is that, if one still denotes $\Sigma$ the singular set appearing in Theorem \ref{thm basic regularity}(ii), the set $\Sigma \setminus {\rm cl}(E)$ is \emph{empty} in all dimensions $n \le 6$ (thus, in particular, in the physical dimension $n=2$), so that $K \setminus {\rm cl}(E)$ is a smooth (in fact, analytic) stable minimal hypersurface of $\Omega \setminus {\rm cl}(E)$ in such cases; see \cite[Theorem 1.5]{kms3}.

\end{remark}

\subsection{Main results}\label{result} We start making precise the notion of collapsing.

\begin{definition} \label{def:collapsed}
A generalized minimizer $(K,E) \in \KK$ of $\psi(\e)$ is {\bf collapsed} if $K \setminus \pa E \neq \emptyset$. It is {\bf exteriorly collapsed} if $K \setminus \cl(E) \neq \emptyset$.
\end{definition}

\begin{theorem}[Convex hull property] \label{thm:main}
If $(K,E) \in \KK$ is an exteriorly collapsed generalized minimizer of $\psi(\e)$, then $K \subset \conv(W)$.
\end{theorem}

\begin{remark} \label{rmk:exterior collapsing}
Theorem \ref{thm:main} can be regarded as an extension to the capillarity model of the classical convex hull property valid in the context of (generalized) minimal surfaces. It is worth noticing that the assumption of exterior collapsing is necessary in this setting. It is easy to construct examples of non-collapsed minimizers of $\psi(\e)$ for which the convex hull property fails: for instance, in the situation of Figure \ref{fig whatisthisabout}(a), as soon as the volume parameter $\e$ is slightly increased, it is clear that part of the corresponding minimizer lies outside of the convex hull of the boundary data.
\end{remark}

Theorem \ref{thm:main} will be proved in two steps, which are of independent interest, and for this reason we record them in two separate statements. First, we show that exterior collapsing enforces a sign condition on the multiplier $\l$ appearing in the Euler--Lagrange equation \eqref{stationary main}. Then, we establish the validity of the convex hull property for a solution to \eqref{stationary main} in the regime $\l \leq 0$.

\begin{theorem}\label{thm lambda is negative}
Let $(K,E) \in \KK$ be an exteriorly collapsed generalized minimizer of $\psi(\e)$. Then, the Lagrange multiplier $\l$ in the Euler-Lagrange equation \eqref{stationary main} satisfies $\l \leq 0$.
\end{theorem}

\begin{theorem} \label{thm convex hull}
Suppose that a pair $(K,E) \in \KK$ satisfies the identity \eqref{stationary main} with $\l \leq 0$. Then, $K$ is contained in the convex hull $\conv(W)$. Moreover, if $\l < 0$, then $K \subset \conv(W \cap \cl(K))$.
\end{theorem}

Theorem \ref{thm:main} is then an immediate corollary of Theorems \ref{thm lambda is negative} and \ref{thm convex hull}. Observe that the validity of the strict inequality $\l < 0$ produces a stronger version of the convex hull property compared to the classical result for minimal surfaces. The proof of Theorem \ref{thm convex hull} is obtained by adapting the argument typically used to establish the convex hull property for stationary varifolds (roughly, the case $\l = 0$ of Theorem \ref{thm convex hull}), see \cite[Theorem 19.2]{SimonLN}. Proving Theorem \ref{thm lambda is negative} is more challenging, and is based on the following geometric idea. Given an exteriorly collapsed generalized minimizer $(K,E)$, we define a one-parameter family of competitors $\{(K_t,E_t)\}_{t > 0}$ with $(K_t,E_t) \in \KK$ and $|E_t| = \e$ by first adding some positive volume $t$ near a point in the collapsed region $K \setminus \cl(E)$, and then restoring the volume constraint by ``locally pushing inwards'' $E$ at a point in $\pa^*E$; see Figure \ref{fig openup} below. Since $K \setminus \cl(E)$ and $\pa^*E$ have, respectively, $0$ and $\l$ mean curvature, we find $\F(K_t,E_t)=\F(K,E)-\l\,t+{\rm O}(t^2)$, so that $\l\le 0$ follows by letting $t\to 0^+$, \emph{provided} we can show that $\F(K,E)\le\F(K_t,E_t)$. This inequality requires a dedicated argument. Indeed, we only know that $(K,E)$ minimizes the relaxed energy $\F$ with respect to its diffeomorphic images, and in fact $K_t$ cannot be represented as the image of $K$ through a map, let alone through a diffeomorphism. To prove $\F(K,E)\le\F(K_t,E_t)$, we will instead approximate $(K_t,E_t)$ by a sequence of open sets $F_j$ in $\E$ having volumes $|F_j|$ converging to $\e$ as $j \to \infty$. Since $\F(K,E) = \psi(\e)$, and $\psi(\cdot)$ is lower semicontinuous on $\left( 0, \infty \right)$, we will obtain the desired inequality if we are able to enforce that the $\H^n$ measure of the boundaries $\pa F_j$ in $\Om$ is not larger than $\F(K_t,E_t)$ for large $j$. This construction is the main technical difficulty of this note, and it exploits in a crucial way the regularity properties of $K$ as described in Theorem \ref{thm basic regularity}. The details are discussed in Lemma \ref{lemma opening}.

\section{Proof of Theorem \ref{thm lambda is negative} } \label{sec:proof1} We start with a simple lemma on orientability, which allows to strengthen conclusion (ii) in Theorem \ref{thm basic regularity} from ``there exists $\Sigma\subset K$, closed and with empty interior in $K$, such that $K\setminus\Sigma$ is a smooth hypersurface'' into ``there exists $\Sigma\subset K$, closed and with empty interior in $K$, such that $K\setminus\Sigma$ is a smooth {\bf orientable} hypersurface''. We do not claim that the set $\Sigma$ resulting from this change still satisfies $\H^n(\Sigma\setminus\pa E)=0$.

\begin{lemma}\label{lemma orientabile}
  If $M$ is a smooth hypersurface in $\R^{n+1}$, then there exists a meager closed set $J\subset M$ such that a smooth unit normal vector field to $M$ can be defined on $M\setminus J$.
\end{lemma}

\begin{proof}
  Let $\U$ denote the family of the open sets $A\subset M$ such that a smooth unit normal vector field to $M$ can be defined on $A$. Let $\U^*$ be a non-empty subset of $\U$ which is totally ordered by set inclusion, and set
  \[
  A^*:=\bigcup\{A:A\in\U^*\}\,.
  \]
  Let $\{A_j\}_{j\in\N}\subset\U^*$ be such that
  \[
  A^*=\bigcup_{j\in\N}A_j\,.
  \]
  Since $\U^*$ is totally ordered by set inclusion, we can assume without loss of generality that $A_j\subset A_{j+1}$. By exploiting this monotonicity property we easily prove that $A^*\in\U$, and therefore that $\U^*$ admits an upper bound in the ordering of $\U$. By Zorn's lemma, $\U$ admits a maximal element $A$ with respect to set inclusion. The set $J=M\setminus A$ is closed in $M$. Should $J$ have non-empty interior, we could find $r>0$ and $p\in J$ such that $B_r(p)\cap M\subset J$. Up to decrease $r$, we can entail $B_r(p)\cap M\in\U$, and then that $A\cup(B_r(p)\cap M)\in\U$, against the maximality of $A$ in $\U$.
\end{proof}

Next we show that any $(K,E) \in \KK$ such that $K$ is a smooth orientable hypersurface outside of a meager closed set can be approximated in energy by sets $F\in\E$.

\begin{lemma}\label{lemma opening}
  Let $(K,E)\in\KK$, that is, let $K$ be $\H^n$-rectifiable, relatively closed in $\Om$, and $\C$-spanning $W$, and let $E\subset\Om$ be open with $\Om\cap\cl(\pa^*E)=\Om\cap\pa E\subset K$. Let $\Sigma\subset K$ be a closed set with empty interior relatively to $K$ such that $K\setminus\Sigma$ is a smooth hypersurface in $\Omega$ and such that there exists $\nu\in C^\infty(K\setminus\Sigma;\SS^n)$ with $\nu(x)^\perp=T_x(K\setminus\Sigma)$ for every $x\in K\setminus\Sigma$. Let
  \begin{eqnarray*}
  M_0:=(K\setminus \Sigma)\setminus\cl(E)\,,\qquad M_1:=(K\setminus\Sigma)\cap E\,,\qquad M:=M_0\cup M_1=K\setminus(\Sigma\cup\pa E)\,.
  \end{eqnarray*}
  For every $x \in M$, let $\rho (x) > 0$ be such that $\{x + t\, \rho(x)\, \nu(x)\, \colon \, x \in M \mbox{ and } |t| < 1\}$ is a tubular neighborhood of $M$ in $\R^{n+1}$ (see e.g. \cite[Theorem 6.24]{Lee_DG}). Also, let $\|A_M\|(x)$ be the maximal principal curvature (in absolute value) of $M$ at $x$. Define then a positive function $u:M\to(0,\eta]$ by setting
  \[
  u(x):=\min\Big\{\eta,\frac{\dist(x,\Sigma\cup\pa E\cup W)}2\,,\de\,\rho(x)\,, \frac{\de}{\|A_M\|(x)}\Big\}\,,\qquad\eta\,,\de\in(0,1)\,,
  \]
  where $\eta,\de\in(0,1)$, and let
  \begin{eqnarray*}
    A_0&:=&\Big\{x+t\,u(x)\,\nu(x):x\in M_0\,,0<t<1\Big\}\,,
    \\
    A_1&:=&\Big\{x+t\,u(x)\,\nu(x):x\in M_1\,,0<t<1\Big\}\,,
    \\
    F&:=&A_0\cup \big(E\setminus \cl(A_1)\big)\,;
  \end{eqnarray*}
  see
  \begin{figure}
    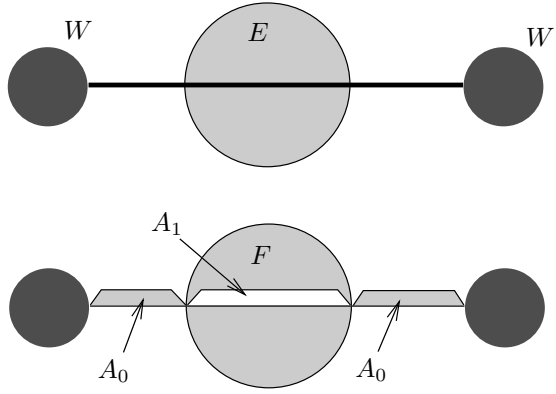\caption{{\small The construction in Lemma \ref{lemma opening}. The part of $K$ outside $\pa^*E$ is denoted by a bold line to recall that in computing $\F(K,E)$ it is counted with multiplicity $2$. Notice that, in principle, $K\setminus\pa^*E$ could intersect $E$.}}\label{fig A0A1}
  \end{figure}
  Figure \ref{fig A0A1}. Then $F\subset\Om$ is open, $\pa F$ is $\H^n$-rectifiable, $K\subset\Om\cap\pa F$ (in particular, $\Om\cap\pa F$ is $\C$-spanning $W$), and
  \begin{equation}
    \label{opening energy}
      \limsup_{\de\to 0^+}\,\limsup_{\eta\to 0^+}\H^n(\Om\cap\pa F)\le\F(K,E)\,.
  \end{equation}
\end{lemma}

\begin{proof}
  {\it Step one}: In this step we prove that
  \begin{eqnarray}
    \label{opening -2}
    &&\mbox{$F$ is open with $F\subset\Om$}\,,
    \\
    \label{opening -1}
    &&K\cup\Big\{x+u(x)\,\nu(x):x\in M\Big\}=\Om\cap\pa F\,.
  \end{eqnarray}
  Since $M_0$ and $M_1$ are relatively open in $M$ and $u$ is positive on $M$, it is easily seen that $A_0$ and $A_1$ are open, and thus that $F$ is open. Let us define a map $\Phi:M\times\R\to\R^{n+1}$ by setting $\Phi(x,t)=x+t\,u(x)\,\nu(x)$, so that
  \begin{equation}
    \label{opening 0}
      A_k=\Phi(M_k\times(0,1))\,,\qquad \Phi(M_k\times\{0,1\})\subset\pa A_k\,,\qquad k=0,1\,.
  \end{equation}
  Since $M\subset\Om$ and $u(x)<\dist(x,W)$ for every $x\in M$, we deduce that
  \begin{equation}
    \label{opening 1.1}
    \Phi(M\times[0,1])\subset\Om\,.
  \end{equation}
  In particular, $F\subset\Omega$ and \eqref{opening -2} is proved. Next we prove that
  \begin{equation}
    \label{fix 4}
    \Om\cap\pa F\,\subset\,\,K\cup\Big\{x+u(x)\,\nu(x):x\in M\Big\}\,.
  \end{equation}
  Since the boundary of the union and of the intersection of two sets is contained in the union of the boundaries, and since the boundary of a set coincides with the boundary of its complement, the inclusion $\pa\cl(A_1)\subset\pa A_1$ gives
  \begin{eqnarray*}
    \Om\cap\pa F&\subset&\Om\cap\Big(\pa A_0\cup\pa[E\setminus\cl(A_1)]\Big)\,\subset\,
    \Om\cap\Big(\pa A_0\cup\pa E\cup\pa [\R^{n+1}\setminus\cl(A_1)]\Big)
    \\
    &=&
    \Om\cap\Big(\pa A_0\cup\pa E\cup\pa \cl(A_1)\Big)\,\subset\,\Om\cap\big(\pa E\cup\pa A_0\cup\pa A_1\big)\,.
  \end{eqnarray*}
  Hence \eqref{fix 4} follows from $\Om\cap\pa E\subset K$, and the fact that, for $k=0,1$,
  \begin{eqnarray*}
 \Om \cap  \pa A_k&\subset&K\cup \Phi(M_k\times\{0,1\})
  \\
  &\subset&K\cup\big\{x+u(x)\,\nu(x):x\in M_k\big\}\,.
  \end{eqnarray*}
  This proves \eqref{fix 4}, so that the proof of \eqref{opening -1} is completed by showing that
  \begin{eqnarray}
    \label{opening 3}
    M\cup \Big\{x+ u(x)\,\nu(x):x\in M\Big\}&\subset&\Om\cap\pa F\,,
    \\
    \label{fix 2}
    \Sigma\setminus\pa E&\subset&\Om\cap\pa F\,,
    \\
    \label{fix 3}
    \Om\cap\pa E&\subset&\Om\cap\pa F\,.
  \end{eqnarray}
  {\it Proof of \eqref{opening 3}}: Since $M_0\cap\cl(E)=\emptyset$, $M_1\subset E$, and $u(x)<\dist(x,\pa E)$ for every $x\in M$, by \eqref{opening 0} we find
  \begin{equation}
    \label{opening 1}
    \Phi\big(M_0\times[0,1]\big)\cap\cl(E)=\emptyset\,,\qquad     \Phi\big(M_1\times[0,1]\big)\subset E\,.
  \end{equation}
  By \eqref{opening 0} and \eqref{opening 1} we find $A_0\cap\cl(E)=\emptyset$ and $A_1\subset E$, so that
  \[
  \Big(\big(\pa A_0\big)\setminus\cl(E)\Big)\,\cup\,\Big(E\cap\pa A_1\Big)\subset\pa F\,.
  \]
  Again by \eqref{opening 0} and \eqref{opening 1} we have
  \begin{equation}
    \label{opening 2}
    \Phi(M_0\times\{0,1\})\subset\pa A_0\setminus\cl(E)\,,\qquad \Phi(M_1\times\{0,1\})\subset E\cap \pa A_1\,,
  \end{equation}
  and \eqref{opening 3} follows by \eqref{opening 1.1} and \eqref{opening 2}. {\it Proof of \eqref{fix 2}}: Since $M_1=(K\setminus\Sigma)\cap E$ and $\Sigma$ has empty interior in $K$, we find that $\cl(M_1)\cap E=K\cap E$. At the same time, $M\subset\Om\cap\pa F$ gives $M_1\cap E\subset E\cap\pa F$ and thus $\cl(M_1)\cap E\subset E\cap\pa F$: hence,
  \[
  \Sigma\cap E\,\,\subset \,\, K\cap E\,\,=\,\,\cl(M_1)\cap E\,\,\subset\,\,\Omega\cap\pa F\,;
  \]
  similarly, $M_0=(K\setminus\Sigma)\setminus\cl(E)$ implies $\cl(M_0)\setminus\cl(E)=K\setminus\cl(E)$, while $M\subset\Om\cap \pa F$ gives $\cl(M_0)\setminus\cl(E)\subset(\pa F)\setminus\cl(E)$, hence
  \[
  \Sigma\setminus\cl(E)\,\,\subset \,\,K\setminus\cl(E)\,\,\subset\,\,(\pa F)\setminus\cl(E)\,,
  \]
  which combined with $\Sigma\subset K\subset\Om$ gives $\Sigma\setminus\cl(E)\subset\Om\cap\pa F$. {\it Proof of \eqref{fix 3}}: since $F$ and $E$ coincide in the complement of $\cl(A_0)\cup\cl(A_1)$, we have
  \[
  \Om\cap\pa E\setminus\big(\cl(A_0)\cup\cl(A_1)\big)\,\,=\,\,\Om\cap\pa F\setminus\big(\cl(A_0)\cup\cl(A_1)\big)\,\,\subset\,\,\Om\cap\pa F\,.
  \]
  Let $y\in\Om\cap\pa E\cap\cl(A_1)$: by \eqref{opening 1}, $y\not\in\Phi(M_1\times[0,1])$ while $A_1=\Phi(M_1\times(0,1))$, so that $y$ is in the closure of $M_1$, and thus of $M$, relatively to $K$. In particular, $y\in\Om\cap\cl(M)\subset\Om\cap\pa F$ thanks to $M\subset\Om\cap\pa F$. Similarly, we can show that $\Om\cap\pa E\cap\cl(A_0)\subset\Om\cap\pa F$ and thus prove \eqref{fix 3}.

  \medskip

  \noindent {\it Step two}: By \eqref{opening -2} and \eqref{opening -1} we immediately deduce all the conclusions except \eqref{opening energy}. To prove \eqref{opening energy} we first notice that thanks to \eqref{opening -1}
  \begin{equation} \label{energy estimate opening1}
  \H^n(\Om\cap\pa F)\le\H^n(K)+\H^n\Big(\big\{x+u(x)\,\nu(x):x\in M\big\}\Big)\,.
  \end{equation}
  Since $\dist(x,\Sigma\cup\pa E\cup W)>0$, $\rho (x) > 0$, and $\|A_M\|(x) < \infty$ for every $x\in M$, we find that the sets
  \[
  M_\eta=\big\{x\in M:u(x)=\eta\big\}=\Big\{x\in M:\dist(x,\Sigma\cup\pa E\cup W)\ge2\eta\,, \rho (x) \geq \frac{\eta}{\delta}\,, \|A_M\|(x) \leq \frac{\delta}{\eta}\Big\}
  \]
  are increasingly converging to $M$ as $\eta\to 0^+$. Moreover, $x\mapsto x+u(x)\,\nu(x)=x+\eta\,\nu(x)$ is smooth on $M_\eta$, so that the area formula gives
  \begin{equation} \label{energy estimate opening2}
  \begin{split}
  \H^n\Big(\big\{x+u(x)\,\nu(x):x\in M_\eta\big\}\Big)&=\int_{M_\eta}\prod_{i=1}^n |1+\eta\,\k_i|
  \\
  &\le(1+\de)^n\,\H^n(M_\eta)\le(1+\de)^n\,\H^n(M)\,,
  \end{split}
  \end{equation}
  where $\k_i$ are the principal curvatures of $M$ with respect to $\nu$. In the limit as $\eta\to 0^+$, the sets $\Phi(M_\eta\times\{1\})=\{x+u(x)\,\nu(x):x\in M_\eta\}$ are increasingly converging to $\Phi(M\times\{1\}) = \{x+u(x)\,\nu(x):x\in M\}$, so that \eqref{energy estimate opening1} and \eqref{energy estimate opening2} yield
  \begin{equation} \label{energy estimate opening final}
  \limsup_{\eta\to 0^+}\H^n(\Om\cap\pa F)\le\H^n(K)+(1+\de)^n\,\H^n(M)\,.
  \end{equation}
  Finally, \eqref{opening energy} follows from \eqref{energy estimate opening final} once we observe that $M=K\setminus(\Sigma\cup\pa E)\subset K\setminus\pa^*E$, so that
  \begin{eqnarray*}
  \H^n(K)+\H^n(M)&=&\H^n(\Om\cap\pa^*E)+\H^n(K\setminus\pa^*E)+\H^n(M)
  \\
  &\le&\H^n(\Om\cap\pa^*E)+2\,\H^n(K\setminus\pa^*E)=\F(K,E)\,,
  \end{eqnarray*}
  as required.
\end{proof}

\begin{proof}
  [Proof of Theorem \ref{thm lambda is negative}] Let $(K,E) \in \KK$ be a generalized minimizer of $\psi(\e)$ satisfying the exterior collapsing condition $K \setminus \cl(E) \neq \emptyset$. The goal is to show that the Lagrange multiplier $\l$ appearing in \eqref{stationary main} must be negative. We introduce the notation
  \begin{eqnarray}\label{not cilindro}
  Q^\nu_r(x)&:=&\Big\{y \in \R^{n+1}\,:\,|(x-y)\cdot\nu|<r\,,\Big|(x-y)-[(x-y)\cdot\nu]\,\nu\Big|<r\Big\}\,,
  \\\label{not disco}
  D^\nu_r(x)&:=&\big\{y \in \R^{n+1}\,:\,|(x-y)\cdot\nu|=0\,,|x-y|<r\big\}\,,
  \end{eqnarray}
  for the cylinder $Q^\nu_r(x)$ with axis along $\nu\in\SS^n$, center at $x$, radius $r$ and height $2\,r$, and for its midsection $D^\nu_r(x)$.

  \smallskip

First recall from \cite[Formula (3.24)]{kms} that the measure $\H^n \mres K$ satisfies a uniform lower density estimate, in the sense that there is a constant $c_0(n)>0$ such that if $x \in K$ then $\H^n(K\cap B_r(x))\ge c_0\,r^n$ for every $B_r(x)\cc\Om$. The above estimate applied with $x \in K \setminus \cl(E)$ and $0 < r < \min\{ {\rm dist}(x, \pa\Om), {\rm dist}(x,\cl(E))\}$ implies that $\H^n(K \setminus \cl(E)) > 0$. By Theorem \ref{thm basic regularity}-(ii), there exists $B_{2\,r_1}(x_1)\cc\Om\setminus\cl(E)$ with $x_1\in K$ such that $K\cap B_{2\,r_1}(x_1)$ is a smooth embedded minimal surface. Let us set
  \[
  Q_1=Q^{\nu_1}_{r_1}(x_1)\,,\qquad D_1=D^{\nu_1}_{r_1}(x_1)\,,
  \]
  where $\nu_1$ is a unit normal to $K$ at $x_1$, and observe that $Q_1 \subset B_{2\,r_1}(x_1)$. Upon further decreasing the value of $r_1$, there exists a smooth solution to the minimal surfaces equation $u_1:\cl(D_1)\to\R$ such that
  \begin{equation} \label{e:primo cilindro}
  K\cap \cl(Q_1)=\big\{z+u_1(z)\,\nu_1:z\in \cl(D_1)\big\}\,, \qquad \max_{\cl(D_1)}|u_1|\le \frac{r_1}2\,.
  \end{equation}
  Next we pick a smooth function $v_1:\cl(D_1)\to\R$ with
  \begin{equation} \label{prima variazione}
  v_1=0\quad\mbox{on $\pa D_1$}\,,\qquad
  v_1>0\quad\mbox{on $D_1$}\,,\qquad\int_{D_1}v_1=1\,,
  \end{equation}
  and for $t > 0$ we define an open set $G_1^t$ by
  \begin{equation} \label{bump}
  G_1^t=\Big\{z+h\,\nu_1:z\in D_1\,,u_1(z)<h<u_1(z)+t\,v_1(z)\Big\}\,.
  \end{equation}
  For $t$ sufficiently small (depending only on $r_1$ and on the choice of $v_1$) we have that $G_1^t\subset Q_1$ with
  \begin{equation}
  \pa G_1^t\cap\pa Q_1=K\cap \pa Q_1=\big\{z+u_1(z)\,\nu_1:z\in \pa D_1\big\}\,,
  \end{equation}
  and
  \begin{equation}
  K\cap \cl(Q_1)\subset\pa G_1^t\,.
  \end{equation}
  Moreover we easily see that
  \begin{equation}
    \label{collapse G1 vol and per}
      |G_1^t|=t\,,\qquad \H^n(\pa G_1^t)=\H^n(Q_1\cap\pa G_1^t)=2\,\H^n(K\cap Q_1) + {\rm O}(t^2) \quad \mbox{as $t \to 0^+$}\,,
  \end{equation}
  where we have used $\int_{D_1} v_1=1$, $v_1=0$ on $\pa D_1$, and the fact that $u_1$ solves the minimal surfaces equation. Next, we perform an analogous construction at a point $x_2\in\Om \cap \pa^*E$, taking once again advantage of Theorem \ref{thm basic regularity}(ii). More precisely, if we let $\nu_2$ denote the exterior unit normal vector to $\pa^*E$ at $x_2$, we find a cylinder $Q_2 = Q_{r_2}^{\nu_2}(x_2)$ with mid-section $D_2 = D_{r_2}^{\nu_2}(x_2)$ and with $\dist(Q_1,Q_2)>0$, and a smooth function $u_2:\cl(D_2)\to\R$ with
  \begin{eqnarray} \label{secondo cilindro filling}
    E \cap \cl(Q_2)&=&\Big\{z+h\,\nu_2:z\in \cl(D_2)\,,-r_2\leq h<u_2(z)\Big\}\,, \label{secondo cilindro set}
    \\
    K\cap\cl(Q_2)&=&\pa E\cap\cl(Q_2)=\Big\{z+u_2(z)\,\nu_2:z\in \cl(D_2)\Big\}\,,
  \end{eqnarray}
  and
  \begin{equation}
    \label{collapse u2 cmc}
      -\Div\Big(\frac{\nabla u_2}{\sqrt{1+|\nabla u_2|^2}}\Big)=\l\quad\mbox{on $D_2$}\,,\qquad
  \max_{\cl(D_2)}|u_2|\le \frac{r_2}2\,.
  \end{equation}
  We choose a smooth function $v_2:\cl(D_2)\to\R$ with
  \begin{equation} \label{seconda variazione}
  v_2=0\quad\mbox{on $\pa D_2$}\,,\qquad
  v_2>0\quad\mbox{on $D_2$}\,,\qquad\int_{D_2}v_2=1\,,
  \end{equation}
  and then define an open set $G_2^t$ by setting
  \begin{equation} \label{cave}
  G_2^t=\Big\{z+h\,\nu_2:z\in D_2\,,u_2(z)-t\,v_2(z) < h < u_2(z)\Big\}\,.
  \end{equation}
  For $t$ small enough (depending only on $r_2$ and on the choice of $v_2$) we have that $G_2^t\subset E \cap Q_2$, with
  \begin{equation}
  \pa G_2^t\cap\pa Q_2=K\cap \pa Q_2=\big\{z+u_2(z)\,\nu_2:z\in \pa D_2\big\}\,.
  \end{equation}
Furthermore, if we let $Y$ denote the closed set
\begin{equation} \label{d:S_surface}
Y = \left\lbrace z + (u_2(z) - t v_2(z))\, \nu_2 \, \colon \, z \in \cl(D_2) \right\rbrace\,,
\end{equation}
  it is easily seen that for $t<t_0$
  \begin{equation}
    \label{collapse G2 vol and per}
      |G_2^t|=t\,,\qquad \H^n(Y) = \H^n (Y \cap Q_2) = \H^n(\pa E \cap Q_2) - \lambda\, t + {\rm O}(t^2)\,,
  \end{equation}
  where we have used $\int_{D_2} v_2=1$, $v_2=0$ on $\pa D_2$, and \eqref{collapse u2 cmc}.\\

\smallskip

   Now set
  \begin{eqnarray} \label{variation1K}
  K_t&:=&\Big(K\setminus\big(Q_1\cup Q_2\big)\Big)\cup\pa G_1^t\cup Y\,,
  \\ \label{variation1E}
  E_t&:=&\Big(E\setminus\cl(G_2^t)\Big)\cup G_1^t\,;
  \end{eqnarray}
  see
  \begin{figure}
  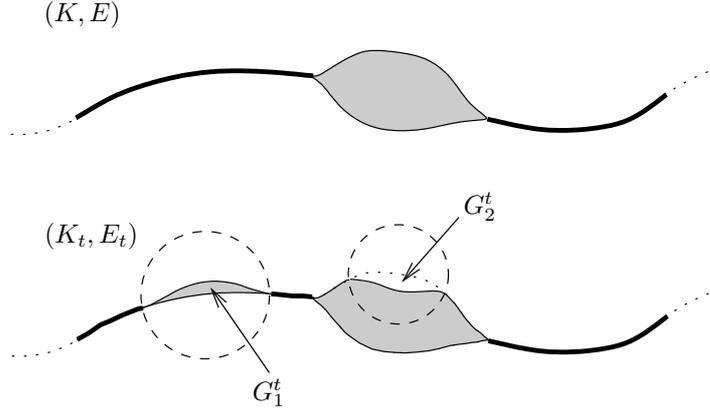\caption{{\small The competitors $(K_t,E_t)$ used in proving that if exterior collapsing occurs for $(K,E)$, then the Lagrange multplier $\l$ is non-positive. Multiplicity two regions are depicted in bold, the sets $E$ and $E_t$ in gray. The competitor $(K_t,E_t)$ is obtained by adding a volume $t$ near a point of $K\setminus\cl(E)$ by bulging one of the two available sheets, at an area cost of ${\rm O}(t^2)$ (as $K\setminus\cl(E)$ is minimal); and by restoring the total volume by pushing inwards $E$ at a point in $\pa^*E$, at an area cost of $-\l\,t+{\rm O}(t^2)$.}}\label{fig openup}
  \end{figure}
  Figure \ref{fig openup}. We claim that the following holds:
  \begin{align} \label{e:compact outside}
  K_t \setminus Q_2 &\supset K \setminus Q_2\,, \\  \label{e:boundary outside}
  \pa E_t \cap \Om \setminus (\cl (Q_1) \cup \cl(Q_2)) &= \pa E \cap \Om \setminus (\cl (Q_1) \cup \cl(Q_2)) \,,\\ \label{e:boundary preciso}
  \pa E_t &= \pa \left[ E \setminus \cl (G_2^t) \right] \cup \pa G_1^t\,,\\ \label{e:boundary inside Q1}
  \pa E_t \cap \cl(Q_1)& = \pa G_1^t\,,\\ \label{e:boundary in Q2}
  \pa E_t \cap \cl (Q_2) & = Y
  \end{align}

  The inclusion in \eqref{e:compact outside} follows from $K \setminus Q_2 \subset K \setminus (Q_1 \cup Q_2) \cup \pa G_1^t$; \eqref{e:boundary outside} is a consequence of $E_t \setminus (\cl(Q_1) \cup \cl(Q_2)) = E \setminus (\cl(Q_1) \cup \cl(Q_2))$ together with the observation that $\Om \setminus (\cl(Q_1) \cup \cl(Q_2))$ is an open set; to prove \eqref{e:boundary preciso}, it suffices to observe that $\pa \left[ E \setminus \cl(G_2^t)  \right] \subset \cl(E)$ whereas $\pa G_1^t \subset \cl(Q_1) \subset B_{2\,r_1}(x_1) \cc \Om \setminus \cl(E)$; \eqref{e:boundary inside Q1} then follows immediately from \eqref{e:boundary preciso}. To prove $\pa E_t \cap \cl (Q_2) \subset Y$ (the other inclusion being trivial), we proceed as follows. First, we deduce from \eqref{e:boundary preciso} that $\pa E_t \cap \cl (Q_2) = \pa \left[ E \setminus \cl (G_2^t) \right] \cap \cl (Q_2)$. Then, we notice that $\pa \left[ E \setminus \cl (G_2^t) \right] \cap \pa Q_2 \subset K \cap \pa Q_2 \subset Y$. Finally, suppose that $x \in \pa \left[ E \setminus \cl(G_2^t)  \right] \cap Q_2$, so that there exists a sequence $\{x_j\}_{j=1}^{\infty}$ such that $x_j \in E \setminus \cl(G_2^t) \cap Q_2$ and $x_j \to x$. In particular, we have $x_j = z_j + h_j\, \nu_2$, where $z_j \in D_2$ and $-r_2 < h_j < u_2(z_j) - t v_2(z_j)$. By compactness, and using the continuity of the functions $u_2$ and $v_2$, we have that, possibly along a (not relabeled) subsequence, $z_j \to z_\infty \in \cl(D_2)$, and $h_j \to h_{\infty} \in \left[ -r_2, u_2(z_\infty) - t v_2(z_\infty) \right]$, so that $x = z_\infty + h_\infty\, \nu_2$. But then it has to be $h_\infty = u_2(z_\infty) - t\, v_2(z_\infty)$, otherwise $x \in E \setminus \cl(G_2^t) \subset E_t$. This shows that $x \in Y$, thus completing the proof of \eqref{e:boundary in Q2}.

\smallskip

Next, we claim that $(K_t,E_t) \in \KK$, and that
  \begin{equation}
    \label{collapse stime KtEt}
    |E_t|=|E|=\e\,,\qquad \F(K_t,E_t)=\F(K,E)-\l\,t+{\rm O}(t^2)\,.
  \end{equation}

First, it is clear that $E_t \subset \Om$ is open, and that $K_t \subset \Om$ is a relatively closed and $\H^n$-rectifiable set in $\Om$. Moreover, $K_t$ is $\C$-spanning $W$. To see this, first observe that by \eqref{e:compact outside} any curve $\gamma \in \C$ with $\gamma \cap ( K \setminus Q_2 ) \neq \emptyset$ must intersect $K_t$. If, on the other hand, $\gamma \cap (K \setminus Q_2) = \emptyset$, then necessarily $\gamma \cap K \cap Q_2 \neq \emptyset$ because $K$ is $\C$-spanning $W$. In turn, this implies that $\gamma \cap \pa E \cap \cl(Q_2) \neq \emptyset$, and thus also $\gamma \cap Y \neq \emptyset$ as a consequence of \cite[Lemma 2.3]{kms} since $Y$ is a diffeomorphic image of $\pa E \cap \cl (Q_2)$. Finally, $\Om \cap \pa E_t \subset K_t$ follows immediately from \eqref{e:boundary outside}, \eqref{e:boundary inside Q1}, and \eqref{e:boundary in Q2}. The volume identity in \eqref{collapse stime KtEt} is deduced from the volume identities in \eqref{collapse G1 vol and per} and \eqref{collapse G2 vol and per} given that $G_1^t$ and $E$ are disjoint. We can then proceed with the proof of the second equation in \eqref{collapse stime KtEt}. Using the analogous of \eqref{e:boundary outside} for the reduced boundary together with \eqref{e:boundary inside Q1} and \eqref{e:boundary in Q2}, and then applying \eqref{collapse G1 vol and per} and \eqref{collapse G2 vol and per} we obtain
\begin{equation} \label{perimeter 1}
\begin{split}
\H^n(\Om \cap \pa^*E_t) & = \H^n(\Om \cap \pa^*E \setminus (\cl(Q_1) \cup \cl (Q_2))) + \H^n(\pa G_1^t) + \H^n(Y) \\
& = \H^n(\Om \cap \pa^*E) + 2\, \H^n((K \setminus \pa^*E) \cap Q_1) -\lambda \, t + {\rm O}(t^2) \,,
\end{split}
\end{equation}
whereas
\begin{equation} \label{perimeter 2}
2\, \H^n(K_t \setminus \pa^*E_t) = 2\, \H^n((K \setminus \pa^*E) \setminus Q_1 )\,.
\end{equation}
The second part of \eqref{collapse stime KtEt} is then obtained by summing \eqref{perimeter 1} and \eqref{perimeter 2}.

\smallskip

Finally, we claim that there exists a closed set $\Sigma_t \subset K_t$ with empty interior relatively to $K_t$ and such that $K_t \setminus \Sigma_t$ is a smooth orientable hypersurface in $\Om$. Indeed, in the construction of $K_t$ from $K$, we may have increased $\Sigma$, at most, by adding to it the closed sets $\{z+u_k(z)\,\nu_k\,\colon\,z\in\pa D_k\}$, which have definitely empty interiors relatively to $K_t$.

  \smallskip

   Therefore we can apply Lemma \ref{lemma opening} to $(K_t,E_t)$ to find a sequence $\{F_j\}_j\subset\E$ such that $\Om\cap\pa F_j$ is $\C$-spanning $W$, with
  \begin{equation}
  F_j\to E_t\quad\mbox{in $L^1(\R^{n+1})$}\,,\qquad\limsup_{j\to\infty}\H^n(\Om\cap\pa F_j)\le\F(K_t,E_t)\,.
  \end{equation}
  Since $|F_j|\to |E_t|=\e$ as $j\to\infty$ and $\psi$ is lower semicontinuous on $(0,\infty)$ (see \cite[Theorem 1.9]{kms}), we conclude that
  \begin{eqnarray*}
  \F(K,E)&=&\psi(\e)\le\liminf_{j\to\infty}\psi\big(|F_j|\big)
  \le\limsup_{j\to\infty}\H^n(\Om\cap\pa F_j)\\
 & \leq &\F(K_t,E_t)=\F(K,E)-\l\,t+{\rm O}(t^2)\,,
  \end{eqnarray*}
  thanks to \eqref{collapse stime KtEt}. By letting $t\to 0^+$ we find that it must be $\l\le 0$, thus completing the proof.
\end{proof}

\section{Proof of Theorem \ref{thm convex hull}} \label{sec:proof2}

\begin{proof}[Proof of Theorem \ref{thm convex hull}]

Let $(K,E) \in \KK$ be such that
  \begin{equation}
    \label{basic stationary main star}
    \l\,\int_{\pa^*E}X\cdot\nu_E\,d\H^n=\int_{\pa^*E}\Div^K\,X\,d\H^n+2\,\int_{K\setminus\pa^*E}\Div^K\,X\,d\H^n\,,
  \end{equation}
  with $\l\le 0$ for every $X\in C^1_c(\R^{n+1};\R^{n+1})$ with $X \cdot \nu_{\Omega} = 0$ on $\pa\Omega$. We then prove that $K \subset \conv(W)$ if $\l = 0$, and $K \subset \conv(W \cap \cl(K))$ if $\l < 0$. The first claim is classical: indeed, if \eqref{basic stationary main star} holds with $\l = 0$ then the varifold $V$ supported on $K$ with multiplicity $\theta = 1$ on $\Om \cap \pa^*E$ and $\theta = 2$ on $K \setminus \pa^*E$ is stationary in $\Om = \R^{n+1} \setminus W$. The result is then a straightforward consequence of \cite[Theorem 19.2]{SimonLN}. We are left with the case $\l < 0$. In order to ease the notation, we set $Z := \conv(W \cap \cl(K))$, and, denoting $u(x):=\dist(x,Z)$, we consider the test field
\begin{equation} \label{e:test field completo}
X(x):=\chi(x)\,\g(u(x))\,\nabla u(x)\,,
\end{equation}
where $\g$ is a non-negative smooth function on $[0,\infty)$ with $\g=0$ on an interval $[0,2\eta)$ and $\g'\ge 0$ everywhere, and $\chi$ is a smooth cut-off function with $0 \leq \chi \leq 1$ and
\[
\chi(x) =
\begin{cases}
1 &\mbox{on $I_\s(K \setminus U_\eta(Z))$} \\
0 &\mbox{on $I_\s(W) \cup (\R^{n+1} \setminus B_R(0))$}\,.
\end{cases}
\]
Here $0 < \s \ll \eta$, and $B_R(0)$ is a large ball containing $K \cup W$. Observe that the function $\chi$ is well-defined. Indeed, the definition of $Z$ implies that the set $K \setminus U_\eta(Z)$ is closed in $\R^{n+1}$, so that ${\rm dist}(K \setminus U_\eta(Z), W) \geq 3\,\s > 0$, and thus the closed sets $I_\s(K \setminus U_\eta(Z))$ and $I_\s(W)$ are disjoint. Since $X = 0$ both in a neighborhood of $W$ and outside of $B_R(0)$, $X$ is admissible in \eqref{basic stationary main star}. Furthermore,
\begin{equation} \label{e:test field K}
X(x) = \g(u(x))\,\nabla u(x) \qquad \mbox{in a neighborhood of $K$}\,.
\end{equation}

Hence, by $|\nabla u|=1$ we compute
\begin{eqnarray*}
\nabla X=\g'(u)\,\nabla u\otimes\nabla u+\g(u)\,\nabla^2 u \qquad &\mbox{in a neighborhood of $K$}\,,&
\\
\Div X=\g'(u)+\g(u)\,\Delta u \qquad &\mbox{in a neighborhood of $K$}\,,&
\\
\Div^KX=\g'(u)\,(1-(\nabla u\cdot\nu)^2)+\g(u)\,\big(\Delta u-\nabla^2 u[\nu,\nu]\big) \qquad &\mbox{$\H^n$-a.e. on $K$}\,,&
\end{eqnarray*}
where $\nu(x)$ is a unit normal vector to $K$ at $x$, for every $x\in K$ such that the approximate tangent plane $T_xK$ exists. Since $u$ is convex (distance from a convex set) we have $\Delta u\ge0$, $\Delta u-\nabla^2 u[\nu,\nu]\ge0$, and thus $\Div^KX\ge0$ $\H^n$-a.e. on $K$. By  \cite[Chapter 16]{maggiBOOK}, for a.e. $\eta>0$, $E\setminus I_\eta(Z)$ is a set of finite perimeter with
\begin{eqnarray*}
  \pa^*(E\setminus I_\eta(Z))=\big((\pa^*E)\setminus I_\eta(Z)\big)\cup\big(E\cap \pa^*I_\eta(Z)\big)\qquad\mbox{modulo $\H^n$}\,,
\end{eqnarray*}
and
\begin{eqnarray*}
  &&\nu_{E\setminus I_\eta(Z)}=\nu_E\,,\qquad\hspace{0.4cm}\mbox{$\H^n$-a.e. on $(\pa^*E)\setminus I_\eta(Z)$}\,,
  \\
  &&\nu_{E\setminus I_\eta(Z)}=-\nabla u\,,\qquad\mbox{$\H^n$-a.e. on $E\cap\pa^*I_\eta(Z)$}\,.
\end{eqnarray*}
By \eqref{basic stationary main star}, $\Div^KX\ge0$, and by applying the divergence theorem on $E\setminus I_\eta(Z)$ we find that
\begin{eqnarray*}
0&\le&\l\int_{\pa^*E}X\cdot\nu_E=\l\,\int_{(\pa^*E)\setminus I_\eta(Z)} (\g(u)\,\nabla u)\cdot\nu_E\\
&=&
\l\,\Big\{\int_{E\setminus I_\eta(Z)} \Div (\g(u)\,\nabla u)-\int_{E\cap\pa^*\,I_\eta(Z)}(\g(u)\,\nabla u)\cdot(-\nabla u)\Big\}
\\
&=&
  \l\,\Big\{\int_{E\setminus I_\eta(Z)}\g'(u)+\g(u)\,\Delta u +\int_{E\cap\pa^*\,I_\eta(Z)}\g(u)\,\Big\}
\end{eqnarray*}
Now we use the condition $\l < 0$. We have
\[
\int_{E\setminus I_\eta(Z)}\g'(u)+\g(u)\,\Delta u +\int_{E\cap\pa^*I_\eta(Z)}\g(u)=0\,,
\]
which implies $|E\setminus I_\eta(Z)|=0$ by the arbitrariness of $\g$, and thus $E\subset Z$ by the arbitrariness of $\eta$. Applying again \eqref{basic stationary main star} we now find
\[
0=\int_{\pa^*E}\Div^KX+2\int_{K\setminus\pa^*E}\Div^KX=2\int_{K\setminus\pa^*E}\Div^KX
\]
which now gives $\H^n(K\setminus I_\eta(Z))=0$ for every $\eta>0$. Thus $K\subset Z$, as claimed.
\end{proof}

\section{Removing assumption (A4)}\label{appendix remove A4} In this final section we show that all the results in \cite{kms} and in the present paper hold without the need of assuming (A4) from Assumption \ref{ass:main}. We notice that (A4) corresponds to (1.12) in \cite{kms}.

\begin{theorem}
  Theorem 1.4, Theorem 1.6 and Theorem 1.9 from \cite{kms} and Theorem 2.6, Theorem 2.8 and Theorem 2.9 from this paper hold under the sole assumption that $W$ and $\C$ satisfy the conditions (A1), (A2) and (A3) stated in Assumption \ref{ass:main}.
\end{theorem}

\begin{proof}
  As noticed in the introductory remarks to the proof of Theorem 1.4 from \cite{kms}, see section 3 of that paper, assumption (A4) (equivalently, \cite[(1.12)]{kms}) is only used in step one of \cite[Proof of Theorem 1.4]{kms} to show that
  \begin{equation}
    \label{step one from thm 1.4 kms}
     \psi(\e)\le 2\,\ell+C\,\e^{n/(n+1)}\,.
  \end{equation}
  Indeed, \eqref{step one from thm 1.4 kms} is proved in \cite{kms} by considering a minimizer $S$ of $\ell$, and then by using as competitors in $\psi(\e)$ the open sets, corresponding to a sequence $\eta_j\to0^+$, obtained by first taking open $\eta_j$-neighborhoods $F_j$ of $S$ in $\Om$ (contributing in the limit $j\to\infty$ to the factor $2\,\ell$ in \eqref{step one from thm 1.4 kms}), and then by adding to these neighborhoods some disjoint balls of volume $\e-|F_j|$ (whose energy contributions are controlled by $C\,\e^{n/(n+1)}$). The role of assumption (A4) is ensuring that the boundaries $\Om\cap\pa F_j$ are $\C$-spanning $W$, and thus that these open set are admissible competitors for $\psi(\e)$.

  We can avoid this difficulty if, rather than working with $\eta$-neighborhoods of $S$, we exploit Lemma \ref{lemma opening} to work with ``unilateral'' open neighborhoods of $S$, which still contain $S$ in their boundary, and thus are automatically $\C$-spanning. More precisely, let us recall that if $S$ is a minimizer of $\ell$, then there exists an $\H^n$-negligible and closed subset $\Sigma^*$ of $S$ such that $S\setminus\Sigma^*$ is a smooth hypersurface (indeed, $S$ is an Almgren minimizer, and therefore it is $\H^n$-a.e. everywhere smooth by the main result in \cite{Almgren76}). By Lemma \ref{lemma orientabile}, we can find a closed meager subset $\Sigma$ of $S$ (with $\Sigma^*\subset\Sigma$) with the property that $S\setminus\Sigma$ is a smooth orientable hypersurface. Therefore we can apply Lemma \ref{lemma opening} with
  \[
  K=S\,,\qquad E=\emptyset\,,
  \]
  to find, for every $\eta,\de\in(0,1)$, an open subset $F$ of $\Om$ such that $\pa F$ is $\H^n$-rectifiable, $S\subset\Om\cap\pa F$, and
  \[
  \limsup_{\de\to 0^+}  \,\limsup_{\eta\to 0^+}\,\H^n\big(\Om\cap\pa F\big)\le \F(S,\emptyset)=2\,\ell\,.
  \]
  Let $\{F_j\}$ correspond to $\de_j\to 0^+$ and $\eta_j\to 0^+$ so that $\limsup_j\H^n(\Om\cap\pa F_j)\le 2\,\ell$, and notice that, by construction, $|F_j|\to 0^+$. We can thus define $E_j=F_j\cup B_{r_j}(p)$ where $r_j$ is such that $|B_{r_j}(p)|=\e-|F_j|$ and where $p$ is such that $\cl[B_{r_j}(p)]$ is disjoint from $W\cup\cl(F_j)$: the resulting sets are competitors for $\psi(\e)$, and their existence implies the validity of \eqref{step one from thm 1.4 kms}.
\end{proof}

\bibliographystyle{is-alpha}
\bibliography{references_mod}
\end{document}